\documentclass[11pt]{article}
\usepackage[a4paper,margin=1in]{geometry}
\usepackage{amsmath,amssymb,amsthm,mathtools}
\usepackage{microtype}
\usepackage{enumitem}
\usepackage[hidelinks]{hyperref}
\usepackage[T1]{fontenc}
\usepackage{lmodern}

\newtheorem{theorem}{Theorem}
\newtheorem{lemma}{Lemma}
\theoremstyle{remark}
\newtheorem{remark}{Remark}

\newcommand{\norm}[1]{\left\lVert #1\right\rVert}

\title{Stability Refinements of the Triangle Inequality in $L^p$ Spaces}
\author{Ruizhou Song\\
\small Student, Wenzhou High School, Zhejiang 325014, China\\
\small \texttt{19016820799@163.com}}
\date{}

\begin{document}
\maketitle

\begin{abstract}
Let $(X,\mu)$ be a measure space and let $1<p<\infty$. We study quantitative stability refinements of Minkowski's inequality
\[
\norm{f+g}_p\le \norm f_p+\norm g_p
\]
for real-valued functions in $L^p(X,\mu)$. We first establish a stability estimate for arbitrary real-valued functions and show that its constant is sharp. We then prove that, for nonnegative functions, the constant can be improved when $p\ge 2$, again to its optimal value. More precisely, if $f,g\ge0$ and $f,g\ne0$, then
\[
\norm{f+g}_p\le \norm f_p+\norm g_p
-c_p\min\{\norm f_p,\norm g_p\}
\norm{\frac{f}{\norm f_p}-\frac{g}{\norm g_p}}_p^{\alpha_p},
\]
where
\[
\alpha_p=\begin{cases}
2,&1<p\le2,\\
p,&2<p<\infty,
\end{cases}
\qquad
c_p=\begin{cases}
\dfrac{p-1}{4},&1<p\le2,\\[4pt]
\dfrac{1-2^{1-p}}{p},&2<p<\infty.
\end{cases}
\]
Both constants are best possible.
\end{abstract}

\section{Introduction and main results}

The triangle inequality in $L^p$ spaces, namely Minkowski's inequality,
\[
\norm{f+g}_p\le \norm f_p+\norm g_p,
\]
is one of the most fundamental inequalities in analysis. This paper studies a question proposed by Professor Gangsong Leng: does the triangle inequality for $L^p$ functions, and especially for nonnegative $L^p$ functions, admit a stability refinement controlled by the difference between the normalized directions of the two functions? Motivated by this question, we obtain the results below.

Equality in Minkowski's inequality essentially requires the two functions to point in the same direction. Thus, in order to quantify the deviation from equality, comparing only $\norm f_p$ and $\norm g_p$ is insufficient. A more natural quantity is the normalized directional distance
\[
D_p(f,g)=\norm{\frac{f}{\norm f_p}-\frac{g}{\norm g_p}}_p.
\]
The purpose of this paper is to bound the deficit in Minkowski's inequality from below in terms of $D_p(f,g)$.

There is a substantial literature on stability refinements of classical inequalities. Aldaz established a stability version of H\"older's inequality and derived related estimates for the triangle inequality~\cite{Aldaz}. Carlen, Frank, Ivanisvili, and Lieb systematically studied refinements of the triangle inequality and complementary forms of Hanner's inequality~\cite{CFIL}. We first prove a stability estimate for general real-valued $L^p$ functions. We then show that, on the nonnegative cone, the constant can be improved substantially for $p\ge2$, and that the improved constant is sharp.

\begin{theorem}[General stability estimate]\label{thm:general}
Let $(X,\mu)$ be a measure space, let $1<p<\infty$, and let $f,g\in L^p(X,\mu)$ be nonzero real-valued functions. Then
\[
\norm{f+g}_p\le \norm f_p+\norm g_p
-C_p\min\{\norm f_p,\norm g_p\}
\norm{\frac{f}{\norm f_p}-\frac{g}{\norm g_p}}_p^{\alpha_p},
\]
where
\[
\alpha_p=\begin{cases}
2,&1<p\le2,\\
p,&2<p<\infty,
\end{cases}
\qquad
C_p=\begin{cases}
\dfrac{p-1}{4},&1<p\le2,\\[4pt]
\dfrac{1}{p2^{p-1}},&2<p<\infty.
\end{cases}
\]
Both constants are best possible.
\end{theorem}

\begin{theorem}[Optimal stability estimate for nonnegative functions]\label{thm:positive}
Let $(X,\mu)$ be a measure space, let $1<p<\infty$, and let $f,g\in L^p(X,\mu)$ be nonzero and nonnegative. Then
\[
\norm{f+g}_p\le \norm f_p+\norm g_p
-c_p\min\{\norm f_p,\norm g_p\}
\norm{\frac{f}{\norm f_p}-\frac{g}{\norm g_p}}_p^{\alpha_p},
\]
where
\[
\alpha_p=\begin{cases}
2,&1<p\le2,\\
p,&2<p<\infty,
\end{cases}
\qquad
c_p=\begin{cases}
\dfrac{p-1}{4},&1<p\le2,\\[4pt]
\dfrac{1-2^{1-p}}{p},&2<p<\infty.
\end{cases}
\]
The constant $c_p$ is best possible in each of the two ranges.
\end{theorem}

\begin{remark}
Theorem~\ref{thm:general} serves as a comparison result for arbitrary real-valued functions. Theorem~\ref{thm:positive} is the main result of the paper. On the nonnegative cone and for $p\ge2$, the constant improves from $1/(p2^{p-1})$ to $(1-2^{1-p})/p$. For example, when $p=4$, it improves from $1/32$ to $7/32$. Thus nonnegativity is not merely a formal restriction; it yields a genuine quantitative strengthening.
\end{remark}

\section{Two lemmas}

Throughout this section, all functions are real-valued and belong to $L^p(X,\mu)$ on the same measure space.

\begin{lemma}\label{lem:smallp}
Let $1<p\le2$, and let $u,v\in L^p(X,\mu)$ satisfy $\norm u_p=\norm v_p=1$. Then
\[
\norm{u+v}_p\le 2-\frac{p-1}{4}\norm{u-v}_p^2.
\]
\end{lemma}

\begin{proof}
A Clarkson--Hanner type estimate, equivalently the Ball--Carlen--Lieb uniform convexity inequality, gives
\[
\norm{\frac{u+v}{2}}_p^2+(p-1)\norm{\frac{u-v}{2}}_p^2
\le \frac{\norm u_p^2+\norm v_p^2}{2}=1.
\]
Hence
\[
\norm{u+v}_p
\le 2\left(1-\frac{p-1}{4}\norm{u-v}_p^2\right)^{1/2}.
\]
Using $\sqrt{1-t}\le1-t/2$ yields
\[
\norm{u+v}_p\le 2-\frac{p-1}{4}\norm{u-v}_p^2.
\]
\end{proof}

\begin{lemma}\label{lem:largep}
Let $2\le p<\infty$, and let $u,v\in L^p(X,\mu)$ satisfy $\norm u_p=\norm v_p=1$.
\begin{enumerate}[label=\textup{(\roman*)}]
\item For arbitrary real-valued $u,v$,
\[
\norm{u+v}_p\le 2-\frac{1}{p2^{p-1}}\norm{u-v}_p^p.
\]
\item If, in addition, $u,v\ge0$, then
\[
\norm{u+v}_p\le 2-\frac{1-2^{1-p}}{p}\norm{u-v}_p^p.
\]
\end{enumerate}
\end{lemma}

\begin{proof}
For part (i), Hanner's inequality for $p\ge2$ gives
\[
\norm{u+v}_p^p+\norm{u-v}_p^p
\le (\norm u_p+\norm v_p)^p+|\norm u_p-\norm v_p|^p=2^p.
\]
Therefore
\[
\norm{u+v}_p\le 2\left(1-\frac{\norm{u-v}_p^p}{2^p}\right)^{1/p}.
\]
Since $(1-t)^{1/p}\le1-t/p$, we obtain
\[
\norm{u+v}_p\le 2-\frac{1}{p2^{p-1}}\norm{u-v}_p^p.
\]

For part (ii), we first prove the scalar inequality
\begin{equation}\label{eq:scalar}
(a+b)^p+(2^{p-1}-1)|a-b|^p\le 2^{p-1}(a^p+b^p)
\end{equation}
for all $a,b\ge0$. By homogeneity, we may assume that $a+b=1$. Put $s=|a-b|\in[0,1]$. Then
\[
a=\frac{1+s}{2},\qquad b=\frac{1-s}{2},
\]
and \eqref{eq:scalar} is equivalent to
\begin{equation}\label{eq:scalar2}
\frac{(1+s)^p+(1-s)^p}{2}\ge 1+(2^{p-1}-1)s^p.
\end{equation}
For $0<s\le1$, define
\[
H(s)=\frac{\frac{(1+s)^p+(1-s)^p}{2}-1}{s^p}.
\]
A direct differentiation gives
\[
H'(s)=\frac{p}{2s^{p+1}}
\left(2-(1+s)^{p-1}-(1-s)^{p-1}\right)\le0,
\]
because $t\mapsto t^{p-1}$ is convex on $[0,\infty)$. Thus $H$ is decreasing on $(0,1]$, and hence
\[
H(s)\ge H(1)=2^{p-1}-1.
\]
This proves \eqref{eq:scalar2}, and therefore \eqref{eq:scalar}.

Applying \eqref{eq:scalar} pointwise to the nonnegative functions $u$ and $v$ and then integrating, we obtain
\[
\norm{u+v}_p^p+(2^{p-1}-1)\norm{u-v}_p^p
\le 2^{p-1}(\norm u_p^p+\norm v_p^p)=2^p.
\]
Consequently,
\[
\norm{u+v}_p
\le2\left(1-\frac{2^{p-1}-1}{2^p}\norm{u-v}_p^p\right)^{1/p}.
\]
Using $(1-t)^{1/p}\le1-t/p$ once more gives
\[
\norm{u+v}_p
\le2-\frac{2^{p-1}-1}{p2^{p-1}}\norm{u-v}_p^p
=2-\frac{1-2^{1-p}}{p}\norm{u-v}_p^p.
\]
\end{proof}

\begin{remark}
The scalar inequality \eqref{eq:scalar} may be regarded as a simple Clarkson-type inequality on the positive cone. Equality holds both when $a=b$ and when $ab=0$.
\end{remark}

\section{Proofs of the theorems}

Set
\[
A=\norm f_p,\qquad B=\norm g_p,\qquad
u=\frac{f}{A},\qquad v=\frac{g}{B}.
\]
Then $\norm u_p=\norm v_p=1$. Without loss of generality, assume that $A\ge B$. We may write
\[
f+g=Au+Bv=(A-B)u+B(u+v).
\]
By Minkowski's inequality,
\[
\norm{f+g}_p
\le(A-B)\norm u_p+B\norm{u+v}_p
=(A-B)+B\norm{u+v}_p.
\]

If $1<p\le2$, Lemma~\ref{lem:smallp} yields
\[
\norm{f+g}_p\le A+B-\frac{p-1}{4}B\norm{u-v}_p^2.
\]
If $2\le p<\infty$, Lemma~\ref{lem:largep}(i) yields
\[
\norm{f+g}_p\le A+B-\frac{1}{p2^{p-1}}B\norm{u-v}_p^p.
\]
Since $B=\min\{A,B\}$, Theorem~\ref{thm:general} follows.

Now suppose in addition that $f,g\ge0$. Then $u,v\ge0$. For $1<p\le2$, we again apply Lemma~\ref{lem:smallp}; for $2\le p<\infty$, we instead apply Lemma~\ref{lem:largep}(ii). This proves Theorem~\ref{thm:positive}.

\section{Sharpness of the constants}

We now show that none of the constants in Theorems~\ref{thm:general} and~\ref{thm:positive} can be increased. Since the asserted inequalities reduce to their unit-norm forms when $\norm f_p=\norm g_p=1$, it suffices to test sharpness in two-dimensional $\ell^p$.

First consider Theorem~\ref{thm:general} for $1<p\le2$. Let
\[
u=\frac{(1,1)}{2^{1/p}},\qquad
v_\varepsilon=
\frac{(1+\varepsilon,1-\varepsilon)}{\bigl((1+\varepsilon)^p+(1-\varepsilon)^p\bigr)^{1/p}},
\]
where $\varepsilon>0$ is sufficiently small. Then $\norm u_p=\norm{v_\varepsilon}_p=1$. A second-order expansion as $\varepsilon\to0^+$ gives
\[
2-\norm{u+v_\varepsilon}_p
\sim \frac{p-1}{4}\norm{u-v_\varepsilon}_p^2.
\]
Thus $(p-1)/4$ cannot be increased. Since this example is nonnegative, it also proves the sharpness of the same constant in Theorem~\ref{thm:positive} for $1<p\le2$.

For Theorem~\ref{thm:general} in the range $2\le p<\infty$, signed vectors are needed. Let
\[
u_t=\bigl((1-t^p)^{1/p},t\bigr),\qquad
v_t=\bigl((1-t^p)^{1/p},-t\bigr),\qquad 0<t<1.
\]
Then $\norm{u_t}_p=\norm{v_t}_p=1$, while
\[
\norm{u_t-v_t}_p=2t,
\qquad
\norm{u_t+v_t}_p=2(1-t^p)^{1/p}.
\]
Hence, as $t\to0^+$,
\[
2-\norm{u_t+v_t}_p
=2-2(1-t^p)^{1/p}
\sim \frac{2}{p}t^p
=\frac{1}{p2^{p-1}}\norm{u_t-v_t}_p^p.
\]
Therefore the constant $1/(p2^{p-1})$ in Theorem~\ref{thm:general} is sharp.

Finally, consider the improved constant in Theorem~\ref{thm:positive} for $2\le p<\infty$. Let
\[
u=(1,0),\qquad
v_t=\bigl((1-t^p)^{1/p},t\bigr),\qquad 0<t<1.
\]
Then $u,v_t\ge0$ and $\norm u_p=\norm{v_t}_p=1$. As $t\to0^+$,
\[
\norm{u-v_t}_p^p=t^p+o(t^p),
\]
whereas
\[
\norm{u+v_t}_p^p
=\left(1+(1-t^p)^{1/p}\right)^p+t^p
=2^p-(2^{p-1}-1)t^p+o(t^p).
\]
It follows that
\[
2-\norm{u+v_t}_p
\sim \frac{1-2^{1-p}}{p}\norm{u-v_t}_p^p.
\]
Thus the constant $(1-2^{1-p})/p$ in Theorem~\ref{thm:positive} cannot be increased.

These examples also demonstrate a genuine distinction between the general and nonnegative settings when $p\ge2$: local cancellation between signed vectors restricts the constant to $1/(p2^{p-1})$, whereas on the nonnegative cone it improves to $(1-2^{1-p})/p$.

\end{document}